\documentclass[11pt,reqno]{amsart}

\usepackage{amssymb}
\usepackage{amsmath}
\usepackage{graphicx}
\usepackage{color}
\usepackage{amsbsy}
\usepackage{verbatim}
\usepackage{amsthm}
\usepackage{hyperref}

\renewcommand{\epsilon}{\varepsilon}

\DeclareMathOperator{\graph}{graph}

\DeclareMathOperator{\proj}{proj}

\theoremstyle{plain}
\newtheorem{theorem}{\textbf{Theorem}}[section]
\newtheorem{lemma}[theorem]{\textbf{Lemma}}

\newtheorem{rmks}[theorem]{\textbf{Remarks}}
\newtheorem{cor}[theorem]{\textbf{Corollary}}
\newtheorem{claim}[theorem]{\textit{Claim}}

\numberwithin{equation}{section}

\def\R{\mathbb{R}}

\def\N{\mathbb{N}}

\def\d{\delta}

\def\t{\tau}

\def\r{\rho}

\def\k{\kappa}

\def\ov{\overline}

\def\lsb{\left[}
\def\rsb{\right]}
\def\lb{\left(}
\def\rb{\right)}
\def\pd{\partial}
\def\cd{\nabla}
\def\eL{\Delta_F}
\newcommand{\XD}{{}^{X}\hspace{-3pt}D}

\def\ba #1\ea {\begin{align} #1\end{align}}
\def\bann #1\eann {\begin{align*} #1\end{align*}}

\newcommand{\inner}[2]{\left\langle #1 \, , \, #2\right\rangle}
\newcommand{\norm}[1]{\left\Vert#1\right\Vert}

\def\ben #1\een {\begin{enumerate} #1\end{enumerate}}
\def\bi #1\ei {\begin{itemize}\renewcommand\labelitemi{--} #1\end{itemize}}

\usepackage{cite}

\title[Type-II singularities of two-convex MCF]{Type-II singularities of two-convex immersed mean curvature flow}
\author{Theodora Bourni and Mat Langford}

\begin{document}

\begin{abstract}
We show that any strictly mean convex translator of dimension $n\geq 3$ which admits a cylindrical estimate and a corresponding gradient estimate is rotationally symmetric. As a consequence, we deduce that any translating solution of the mean curvature flow which arises as a blow-up limit of a two-convex mean curvature flow of compact immersed hypersurfaces of dimension $n\geq 3$ is rotationally symmetric. The proof is rather robust, and applies to a more general class of translator equations. As a particular application, we prove an analogous result for a class of flows of embedded hypersurfaces which includes the flow of two-convex hypersurfaces by the two-harmonic mean curvature.
\end{abstract}

\maketitle

\section{Introduction}

We are interested in hypersurfaces $X:M^n\to\R^{n+1}$ satisfying the {\it translator} equation
\begin{equation}\label{eq:T}\tag{T}
\vec H=T^\perp
\end{equation}
for some constant vector $T\in \R^{n+1}$, where, given a local choice of unit normal field $\nu$, $\vec H=-H\nu$ is the mean curvature vector of the immersion with respect to the choice of mean curvature $H=\mathrm{div}\,\nu$, and $\perp$ denotes the projection onto the normal bundle. We call such immersions \emph{translators}. Up to a time-dependent tangential reparametrization, the family $\{X(\cdot,t)\}_{t\in\R}$ of immersions $X(\cdot,t):M^n\to\R^{n+1}$ defined by $X(x,t):=X(x)+tT$ satisfies the mean curvature flow
\begin{equation}\tag{MCF}\label{eq:MCF}
\pd_tX(\cdot,t)=\vec H(\cdot,t)\,,
\end{equation}
where $\vec H(\cdot,t)$ is the mean curvature vector of $X(\cdot,t)$. We therefore also refer to solutions of \eqref{eq:T} as \emph{translating solutions of the mean curvature flow}. It is well-known that translating solutions arise as blow-up limits of the mean curvature flow about type-II singularities \cite{Hm95a,HuSi99b}. More precisely, if a solution $X:M^n\times[0,T)\to\R^{n+1}$ of \eqref{eq:MCF} has \emph{type-II} curvature blow-up (that is, $\limsup_{t\to T}(T-t)\max_{M^n\times\{t\}}H^2=\infty$)
then there is a sequence of parabolically rescaled solutions of \eqref{eq:MCF} which converge locally uniformly in $C^\infty$ to a (non-trivial) translating solution of \eqref{eq:MCF}. 

Probably the most well-known translator is the Grim Reaper\footnote{So named because, as it translates, it `kills' any compact solution of curve shortening flow which is unfortunate enough to lie in its path.} curve $\Gamma$, which is the graph of the function $x\mapsto -\log\cos x$, $x\in (-\pi/2,\pi/2)$. In dimensions $n\geq 2$, there exists a strictly convex, rotationally symmetric translator asymptotic to a paraboloid, which is commonly referred to as the `bowl' \cite{AW94,CSS07}. The bowl is the unique rotationally symmetric translating complete graph, and the unique translator with finite genus and a single end asymptotic to a paraboloid \cite{MSHS}.  
In a remarkable study of convex ancient graphical solutions of the mean curvature flow, X.-J.~Wang showed that any strictly convex, entire translator in dimension two is rotationally symmetric, and hence the bowl \cite{Wa11}. Moreover, in every dimension $n\geq 3$, he constructed strictly convex, entire examples without rotational symmetry.


In the setting of two-convex (that is, $\kappa_1+\kappa_2>0$, where $\kappa_1\leq\kappa_2\leq\dots\leq\kappa_n$ denote the principal curvatures) mean curvature flow in dimensions $n\geq 3$, the far-reaching theory of Huisken and Sinestrari \cite{HuSi99a,HuSi99b,HuSi09} shows that regions of high curvature are either uniformly convex and cover a whole connected component of the surface, or else they contain regions which are very close, up to rescaling, to cylindrical segments $[-L,L]\times S^{n-1}$. This suggests that the translating blow-up limits which arise at type-II singularities might be rotationally symmetric. We note that this is true (in dimensions $n\geq 3$) for two-convex self-shrinking solutions which arise as blow-up limits of the mean curvature flow with type-I curvature blow-up (that is, $\limsup_{t\to T}(T-t)\max_{M^n\times\{t\}}H^2<\infty$) since the only possibilities are shrinking spheres $S^n_{\sqrt{-2nt}}$ and cylinders $\R\times S^{n-1}_{\sqrt{-2(n-1)t}}$ \cite[Theorem 5.1]{Hu93}. Recently, Haslhofer \cite{Ha} proved that this is true in the embedded case (even in dimension 2), his proof relying crucially on the non-collapsing theory of \cite{An12} and \cite{HK1}. In fact, he shows that any strictly convex, uniformly two-convex translator which is non-collapsing is necessarily rotationally symmetric. In the immersed setting, we no longer have a non-collapsing property; however, by the work of Huisken and Sinestrari \cite{HuSi09}, we have a cylindrical estimate and a corresponding gradient estimate. Motivated by Haslhofer's result and the Huisken--Sinestrari theory, we prove the following.

\begin{theorem}\label{thm:bowl}
Let $X:M^n\to\R^{n+1}$, $n\geq 3$, be a mean convex translator and $C_1<\infty$ a constant such that the following hold:
\begin{enumerate}
\item cylindrical estimate: $|A|^2-\frac{1}{n-1}H^2<0$ 
\item gradient estimate: $|\nabla A|^2\le -C_1\left(|A|^2-\frac{1}{n-1}H^2\right)H^2$
\end{enumerate}
where $A$ is the second fundamental form of $X$.Then $M^n$ is rotationally symmetric.
\end{theorem}

In fact (assuming $T=e_{n+1}$), we need only prove that the blow-down of $M_t^n:=M^n+te_{n+1}$ is the shrinking cylinder $S^{n-1}_{\sqrt{2(n-1)(1-t)}}\times \R$, since this is enough to deduce rotational symmetry of $M^n$ by \S 3--5 of Haslhofer's paper.

We remark  that the cylindrical estimate implies uniform two-convexity, $\kappa_1+\kappa_2\geq \frac{1}{2(n-1)}H$ (see \cite[Lemma 5.1]{HuSi15}).
As a consequence, any type-II blow-up limit of a two-convex mean curvature flow in dimensions $n\geq 3$ is rotationally symmetric (even when the mean curvature flow is only immersed).
\begin{cor}\label{cor:bowl}
Suppose that $X:M^n\to\R^{n+1}$, $n\geq 3$, is a translator which arises as a proper blow-up limit of a two-convex mean curvature flow of immersed hypersurfaces. Then $M^n$ is rotationally symmetric.
\end{cor}

We note that Corollary \ref{cor:bowl} fails in dimension 2 without some additional assumption, such as non-collapsing, to rule out the Grim plane $\R\times \Gamma$. This is in accordance with the type-I case, where the non-embedded Abresch--Langer planes $\R\times \gamma_{k,l}$ can arise \cite{AbLa86}.

We remark that our proof of Theorem \ref{thm:bowl} also works (in dimensions $n\geq 2$) if assumptions (1) and (2) are replaced by
\begin{enumerate}
\item[(1')] cylindrical estimate: $\overline k-\frac{1}{n-1}H<0$ 
and
\item[(2')] gradient estimate: $|\nabla A|^2\le C_1\kappa_1H^3$,
\end{enumerate}
where $\overline k$ denotes the inscribed curvature. By work of Brendle \cite[Theorem~1]{Br15} (see also \cite{HaKl15}) and Haslhofer and Kleiner\footnote{The improved gradient estimate (2') follows from \cite[Corollary 2.7]{HK1} as in the proof of Claim \ref{claim:gradient} in Section \ref{sec:F}.} \cite[Corollary 2.7]{HK1}, these assumptions are met for blow-up limits of type-II singularities of two-convex mean curvature flows of \emph{embedded} hypersurfaces. This provides a slightly different perspective of Haslhofer's result.

Apart from dealing with blow-up limits  of type-II singularities of two-convex mean curvature flows of \emph{immersed} hypersurfaces, a further motivation for removing the (two-sided) non-collapsing assumption in Haslhofer's result was to study translating solutions of more general curvature flows, where (two-sided) non-collapsing will in general not hold. Let $F$ be given by $F(x)=f(\vec\kappa(x))$ for some smooth function $\displaystyle f:\Gamma^n \subset\R^n
\to\R$ of the principal curvatures $\vec\kappa:=(\kappa_1,\dots,\kappa_n)$ defined with respect to some choice of unit normal field $\nu$. Then we can consider solutions $X:M^n\to\R^{n+1}$ of the fully non-linear translator equation
\begin{equation}\tag{FT}\label{eq:FT}
F=-\inner{\nu}{T}
\end{equation}
for some $T\in\R^{n+1}$. We will call the function $f:\Gamma^n\to\R$ \emph{admissible} if $\Gamma^n$ is an open, symmetric cone and $f$ is smooth, symmetric, monotone increasing in each variable and 1-homogeneous. These conditions on $f$ are very natural: Indeed, smoothness and symmetry are needed to ensure that $F$ is smooth, monotonicity ensures that \eqref{eq:FT} is elliptic, and homogeneity ensures that $F$ scales like curvature.

Just as for the mean curvature flow, the family $\{X(\cdot,t)\}_{t\in\R}$ of immersions $X(\cdot,t):M^n\to\R^{n+1}$ defined by $X(x,t):=X(x)+tT$ satisfies, up to a time-dependent tangential reparametrization, the corresponding flow\footnote{We have implicitly assumed orientability of solutions of \eqref{eq:FT} and \eqref{eq:F}; however, if $f$ is an \emph{odd} function, \eqref{eq:FT} and \eqref{eq:F} also admit non-orientable solutions.}
\begin{equation}\tag{F}\label{eq:F}
\pd_tX(\cdot,t)=-F(\cdot,t)\nu(\cdot,t)\,.
\end{equation}
Moreover, if \eqref{eq:F} admits an appropriate Harnack inequality (which is true under very mild concavity assumptions for $f$ \cite{An94c}) then solutions of \eqref{eq:FT} arise as blow-up limits of positive speed solutions of \eqref{eq:F} about type-II singularities in a completely analogous way to the case of mean convex mean curvature flow. If $F$ also admits a strong maximum principle for the Weingarten tensor (which also holds under natural concavity conditions for $f$, see Section \ref{sec:SMP}) then our proof goes through with minor modification, and we obtain a result of the following form (where we denote by $\Gamma_+^{m}$ the positive cone $\Gamma_+^m:=\{(z_1,\dots,z_m)\in\R^m:\min_{1\leq i\leq m}\{z_i\}>0\}$ in $\R^m$).

\begin{theorem}\label{thm:bowlF}
Let $X:M^n\to\R^{n+1}$, $n\geq 3$, be a solution of \eqref{eq:FT}, where $F$ is given by $F(x)=f(\vec\kappa(x))$ for some admissible $f:\Gamma^n\to\R$ such that
\[
\displaystyle \{(0,\hat z):\hat z\in \Gamma_+^{n-1}\}\subset\Gamma^n\subset \Gamma^n_2:=\{z\in \R^n:\min_{1\leq i<j\leq n}\{z_i+z_j\}>0\}
\]
and either
\ben
\item[(i)] $f$ is convex, or 
\item[(ii)] $f$ is concave and the function $f_\ast:\Gamma_+^{n-1}\to\R$ defined by
\[
f_\ast(z^{-1}_2,\dots,z^{-1}_n):=f(0,z_2,\dots,z_n)^{-1}
\]
is concave. 
\een
Suppose that the solution satisfies
\begin{enumerate}
\item a cylindrical estimate, and
\item a corresponding gradient estimate.
\end{enumerate}
Then $M^n$ is rotationally symmetric.
\end{theorem}


The precise form of the assumptions (1) and (2) will be different depending on whether the speed function $f$ is convex or concave. This is made precise in Section \ref{sec:F}. 

As a particular application, we find that translating blow-up limits about type-II singularities of the flows of embedded hypersurfaces studied in \cite{BrHu2} are rotationally symmetric.

\begin{cor}\label{cor:bowlF}
Suppose that $X:M^n\to\R^{n+1}$, $n\geq 3$, is a translator which arises as a blow-up limit of an embedded solution of the flow \eqref{eq:F}, where $F$ is given by $F(x)=f(\vec \kappa(x))$ for some concave admissible $f:\Gamma^n\to\R$ such that
\bi
\item[(i)] $\displaystyle \{(0,\hat z):\hat z\in \Gamma_+^{n-1}\}\subset\Gamma^n\subset \Gamma^n_2:=\{z\in \R^n:\min_{1\leq i<j\leq n}\{z_i+z_j\}>0\}$,
\item[(ii)] $f|_{\pd\Gamma^n}=0$ and 
\item[(iii)] the function $f_\ast:\Gamma_+^{n-1}\to\R$ defined by
\[
f_\ast(z^{-1}_2,\dots,z^{-1}_n):=f(0,z_2,\dots,z_n)^{-1}
\]
is concave.
\ei
Then $X$ is rotationally symmetric.
\end{cor}

We mention that the class of flows to which the corollary applies includes the flow of two-convex hypersurfaces by the two-harmonic mean curvature,
\ba\label{eq:twoharmonicmean}
F:=\left(\sum_{i<j}\frac{1}{\kappa_i+\kappa_j}\right)^{-1}\,,
\ea
and, for $n=3$, the flows of positive scalar curvature hypersurfaces by either the square root of the scalar curvature or the ratio of scalar to mean curvature. Corollary \ref{cor:bowlF} does not include any convex speeds, because, as yet, it is not known if they admit an appropriate gradient estimate (although an appropriate cylindrical estimate was proved in \cite{AnLa14}).




\section{Preliminaries}\label{sec:prelims}

Let $X:M^n\to\R^{n+1}$ be a solution of \eqref{eq:T}. After performing a rotation and a dilation, we can arrange that $T=e_{n+1}$, which we assume from now on. Introducing the height function $h:M^n\to \R$,
\[
h(x):=\inner{X(x)}{e_{n+1}}\,,
\]
we denote
\[
V:=\cd h=\mathrm{proj}_{TM^n}e_{n+1}=e_{n+1}+H\nu\,.
\]
Then the Weingarten curvature $A$ and the mean curvature $H$ satisfy (see, for instance, \cite{Hm95a})
\ba\label{eq:LaplaceA}
-\Delta A=|A|^2A+\cd_VA
\ea
and
\ba\label{eq:LaplaceH}
-\Delta H=|A|^2H+\cd_VH\,.
\ea
A well-known consequence of \eqref{eq:LaplaceA} and the strong maximum principle is the following splitting theorem (see \cite[Theorem 4.1]{HuSi99b} or the appendix).
\begin{theorem}[Splitting Theorem]
Let $X:M^n\to\R^{n+1}$ be a locally weakly convex solution of \eqref{eq:T}. Then, either $\kappa_1>0$ or $\kappa_1\equiv 0$ and $M^n$ splits as an isometric product $M^n\cong \R\times \Sigma^{n-1}$.
\end{theorem}

We next note that a mean convex translator $X:M^n\to\R^{n+1}$ which satisfies the cylindrical estimate must be locally strictly convex. Indeed, 
\begin{align}
|A|^2-\frac{1}{n-1}H^2
={}&
\frac{1}{n-1}\sum_{1<i<j}(\kappa_j-\kappa_i)^2+\frac{n}{n-1}\kappa_1^2-\frac{2}{n-1}\kappa_1H\nonumber\\
\geq{}&
-\kappa_1H\,,\label{eq:roundcrosssection}
\end{align}
so that, wherever the cylindrical estimate holds,
\begin{equation}\label{eq:str.convexity}
\kappa_1\geq -\frac{1}{H}\left(|A|^2-\frac{1}{n-1}H^2\right)>0\,.
\end{equation}
Note also that, for a hypersurface satisfying the weak cylindrical estimate $|A|^2-\frac{1}{n-1}H^2\leq 0$, the only points at which $\kappa_1$ can vanish are the cylindrical points, $\kappa_1=0$, $\kappa_2=\kappa_n$.

Since $M^n$ is smooth and $n\geq 2$, local convexity implies that $M^n$ is the boundary of a convex body \cite{Sack}. In particular, $M^n$ is embedded, so we may drop the parametrization $X$ and identify $M^n$ with its image. A further consequence of convexity and the inequality $\inner{\nu}{e_{n+1}}=H>0$ is the fact that $M^n$ can be written globally as the graph of a function $u:\Omega^n:=\proj_{\R^{n}\times\{0\}}(M^n)\to\R$.

Note also that, applied to the gradient estimate, \eqref{eq:roundcrosssection} yields
\begin{equation}\label{eq:gradient}
\frac{|\cd A|^2}{H^4}\le C_1\frac{\kappa_1}{H}\,.
\end{equation}
Thus, the gradient estimate actually improves wherever $\kappa_1$ is small compared to $H$. 

We conclude this section by recalling the following well-known consequence of gradient estimates for the curvature (cf. \cite[Lemma 6.6]{HuSi09}).
\begin{lemma}\label{lem:convergence}
Let $X:M^n\to\R^{n+1}$ be a mean convex hypersurface and $C<\infty$ a constant such that
\[
\sup_{M^n}\frac{|\cd H|}{H^2}\leq C\,.
\]
Then
\[
\max_{y\in B_{\frac{1}{2CH(x)}}(x)}H(y)\leq 2H(x)\,,
\]
where $B_{\frac{1}{2CH(x)}}(x)$ is the intrinsic ball of radius $\frac{1}{2CH(x)}$ about the point $x$.
\end{lemma}
\begin{proof}
For any unit length geodesic $\gamma:[0,s]\to M$ joining the points $y=\gamma(0)$ and $x=\gamma(s)$, we have
\[
\cd_{\gamma'}H^{-1}\leq C\,.
\]
Integrating yields
\[
H^{-1}(x)-H^{-1}(y)\leq Cs
\]
or, if $s\leq \frac{1}{2CH(x)}$, 
\[
H(y)\leq\frac{H(x)}{1-CH(x)s}\leq 2H(x)\,.
\]
The claim follows.
\end{proof}

\begin{cor}\label{cor:convergence}
Let $(X_j:M_j^n\to\R^{n+1},x_j)_{j\in\N}$ be a sequence of strictly mean convex, weakly locally convex pointed smooth hypersurfaces and $C<\infty$ a constant satisfying
\bann
X_j(x_j)=0,\quad H_j(x_j)=1\quad\text{and}\quad \sup_{M_j^n}\frac{|\cd A_j|}{H_j^2}\leq C\,,
\eann
where, for each $j\in \N$,  $H_j$ and $A_j$ are the mean curvature and second fundamental form, respectively, of $M_j^n$.
Then there exists a weakly locally convex pointed $C^2$ hypersurface $(X_\infty:M_\infty^n\to\R^{n+1},x_\infty)$ such that, after passing to a subsequence, $X_j|_{B_j}:B_j\to\R^{n+1}$ converge locally uniformly in $C^2$ to $X_\infty|_{B_\infty}:B_\infty\to\R^{n+1}$, where $B_j$ denotes the intrinsic ball in $M_j^n$ of radius $(2C)^{-1}$ about the point $x_j$.
\end{cor}

\section{Proof of Theorem \ref{thm:bowl} and Corollary \ref{cor:bowl}}\label{sec:proof}

We begin by noting that the mean curvature goes to zero at infinity.
\begin{lemma}\label{lem:Hto0}
For any sequence of points $X_j\in M^n$ with $\norm{X_j}\to\infty$,
\[
H(X_j)\to 0\,.
\]
\end{lemma}
\begin{proof}
The proof is similar to \cite[Lemma 2.1]{Ha}. Suppose that the lemma does not hold. Then there is a sequence of points $\{X_j\}_{j=1}^\infty\subset M^n$ satisfying $\Vert X_j\Vert\to \infty$ and $\limsup_{j\to\infty}H(X_j)>0$. Passing to a subsequence, we can assume that $\liminf_{j\to\infty}H(X_j)>0$. By translational invariance of \eqref{eq:T}, we can assume, without loss of generality, that $0\in M^n$. Furthermore, after passing to a subsequence, $w_j:=X_j/\Vert X_j\Vert\to w\in S^n$. Consider the sequence $M^n_j:=M^n-X_j$. Since each $M^n_j$ satisfies the translator equation \eqref{eq:T}  and has mean curvature uniformly bounded by 1, it follows from standard regularity theory for solutions of either \eqref{eq:T} \cite{GT} or \eqref{eq:MCF} \cite{EcHu91,Eck} that, after passing to a subsequence, $M^n_j$ converges locally uniformly in $C^\infty$ to a weakly convex translator $M^n_\infty$. We claim that $M^n_\infty$ contains the line $\{sw: s\in \R\}$. First note that the closed convex region $\overline \Omega$ bounded by $M^n$ contains the ray $\{sw:s\geq 0\}$, since it contains each of the segments $\{sw_j:0\leq s\leq s_j\}$, where $s_j:=\Vert X_j\Vert$ and $w_j:=X_j/\Vert X_j\Vert$. By convexity, it also contains the set $\{rsw+(1-r)X_j:s\geq 0, 0\leq r\leq 1\}$ for each $j$. It follows that the closed convex region $\overline\Omega_j$ bounded by $M^n_j$ contains the set $\{rsw-rs_jw_j:s>0, 0\leq r\leq 1\}$. In particular, choosing $s=2s_j$, $\{\vartheta (w-w_j)+\vartheta w:0\leq \vartheta\leq s_j\}\subset \overline\Omega_j$ and, choosing $s=s_j/2$, $\{\vartheta w_j-\vartheta (w-w_j):-s_j/2\leq \vartheta\leq 0\}\subset \overline\Omega_j$. Taking $j\to\infty$, we find $\{sw:s\in\R\}\subset\overline \Omega_\infty$. The claim now follows from convexity of $\overline \Omega_\infty$ since $0\in M^n_\infty$. We conclude that $\kappa_1$ reaches zero somewhere on $M^n_\infty$. By the splitting theorem, the limit splits as an isometric product $M^n_\infty\cong \R\times\Sigma^{n-1}$; in particular, $\kappa_1\equiv 0$. On the other hand, by the strong maximum principle, we must have $H>0$ everywhere (since, by hypothesis, $H(0)>0$). The cylindrical estimate now implies that $\Sigma^{n-1}$ is umbilic (recall \eqref{eq:roundcrosssection}) and hence a round sphere. But this contradicts the fact that $M^n_\infty+te_{n+1}$ satisfies mean curvature flow.
\end{proof}

It follows that $H$ attains a maximum at some point $O$, which we call the `tip' of $M^n$. By translational invariance of \eqref{eq:T}, we can assume, without loss of generality, that $O$ is the origin.

Recall that the gradient field of the height function is given by
\[
V=\proj_{TM^n}(e_{n+1})=e_{n+1}+H\nu\,.
\]
By the translator equation \eqref{eq:T},
\[
\norm{V}^2=1-H^2\,
\]
and, differentiating \eqref{eq:T},
\[
V=-A^{-1}(\cd H)\,.
\]
Moreover, since $A$ is non-degenerate, at any critical point $X$ of $H$ we must have $V(X)=0$ and hence $\nu(X)=-e_{n+1}$. By strict convexity of $M^n$ (recall \eqref{eq:str.convexity}), we conclude that $H$ has precisely one critical point, the origin, and $H\leq H(0)=1$.

Next, observe that
\[
\cd V=HA\,.
\]
Since $A$ is positive definite, it follows from standard ODE theory that we can find, for each $X\in M^n\setminus\{0\}$, a unique integral curve $\phi_X:(0,\infty)\to M^n$ of $ V$ through $X$ such that
\[
\lim_{s\searrow 0}\phi_X(s)=0\quad\text{and}\quad\lim_{s\to \infty}\Vert\phi_X(s)\Vert=\infty\,.
\]
If we parametrize the integral curves by height, so that
\begin{equation}\label{eq:s}
h(\phi_X(s))=\inner{\phi_X(s)}{e_{n+1}}=s\,,
\end{equation}
then we obtain
\bann
\phi_{X}'=\frac{V\circ\phi_X}{1-H^2\circ\phi_X}=\frac{V\circ\phi_X}{\|V\circ \phi_{X}\|^2}\,.
\eann
Note that, by Lemma \ref{lem:Hto0}, the reparametrized curves are still defined on $(0,\infty)$.

We will use the improved gradient estimate \eqref{eq:gradient} to extract a lower bound for $H$ along the flow of $V$.
\begin{lemma}[Lower bound for $H$]\label{lem:Hlowerbound}
There exists $h_0>0$ such that
\[
H(X)\geq \frac{1}{\sqrt{4C_1h(X)}}
\]
for all $X\in  M^n$ with height $h(X)$ at least $h_0$.
\end{lemma}
\begin{proof}
Let $\phi:(0,\infty)\to M$ be an integral curve of $V$ emanating from the tip and parametrized by height. Set $f(s):=H^{-1}(\phi(s))$. Applying the gradient estimate \eqref{eq:gradient}, we obtain
\begin{equation}\label{eq:Hlowerbound1}
(f')^2\leq \frac{|\cd H|^2|\phi'|^2}{H^4}\leq \frac{C_1\kappa_1f}{1-H^2}\,.
\end{equation}
On the other hand,
\bann
-\nabla_{\phi'}H=A(\phi',V)=\frac{A(V,V)}{\norm{V}^2}\geq \kappa_1\,.
\eann
That is,
\begin{equation}\label{eq:Hlowerbound2}
\kappa_1f\leq \frac{f'}{f}\,.
\end{equation}
Putting \eqref{eq:Hlowerbound1} and \eqref{eq:Hlowerbound2} together yields
\[
(f^2)'\leq \frac{2C_1}{1-H^2}\,.
\]
By Lemma \ref{lem:Hto0}, there exists $h_1>0$ such that $H(X)\leq \frac{1}{\sqrt{3}}$ for all $X\in M^n$ with height $h(X)$ at least $h_1$. Thus, for any $s\ge h_1$, we obtain $(f^2)'\leq 3C_1$. Integrating this between $h_1$ and $s$ yields
\[
\frac{1}{H^2(\phi(s))}\le 3C_1(s-h_1)+\frac{1}{H^2(\phi(h_1))}\le 4C_1 s
\]
with the last inequality being true provided that $s$ is large enough; in particular, $s\ge \max\{h_1, (C_1 H^2(\phi(h_1)))^{-1}\}$. Since $\phi$ is parametrized by height, the lemma then follows with $h_0=\max\{h_1, (C_1 H^2(h_1))^{-1}\}$, where $H(h_1):=\min \{H(X):X\in M^n, h(X)=h_1\}>0\,$.
\end{proof}

Next, we derive a lower bound for the `girth' of $M^n$. This estimate plays a key role in obtaining an upper bound for $H$.

\begin{lemma}[Girth estimate]\label{lem:girthestimate}
There exists $h_0<\infty$ such that 
\[
\Vert\proj_{\R^n\times\{0\}}X\Vert\ge\sqrt {\frac{h(X)}{16C_1}}
\] 
for any $X\in M^n$ with height $h(X)$ at least $h_0$.
\end{lemma}
\begin{proof}
The idea of the proof is the following: If the claim did not hold, then, writing $M=\graph u$, there would be a point where $u$ is both large and has large gradient (compared to $\sqrt h$). But this would contradict the lower bound $H\gtrsim h^{-\frac{1}{2}}$ since $H=\frac{1}{\sqrt{1+|Du|^2}}$.

By Lemma \ref{lem:Hlowerbound}, there exists $h_1>0$ such that
\ba\label{eq:Hlowerbond1}
H(X)\geq \frac{1}{\sqrt{4C_1h(X)}}
\ea
for all $X\in  M^n$ with height $h(X)\geq h_1$. Suppose, contrary to the claim, that there is a point $X=(x,u(x))\in M^n$ with $h(X)\geq h_0:=2h_1$ but
\[
\frac{h(X)}{\norm{x}}>\sqrt{16C_1h(X)}\,.
\]
Set $\ell:=\norm{x}$ and consider the curve $\gamma:[0,\ell]\to M^n$ given by $\gamma(s):=(\hat\gamma(s),u(\hat\gamma(s)))$, where $\hat\gamma(s):=s\frac{x}{\ell}$ is the straight line in $\Omega$ joining 0 and $x$. Slightly abusing notation, set $h(s):=h(\gamma(s))$. Then
\[
h'=\cd_{\gamma'}h\leq \norm{V}\sqrt{1+(D_{\hat\gamma'}u)^2}\leq \sqrt{1-H^2}\sqrt{1+|Du|^2}=\frac{\sqrt{1-H^2}}{H}\leq\frac{1}{H}\,.
\]
Let $s_1\in[0,\ell]$ be the point at which $h(\gamma(s_1))=h_1$. By the mean value theorem, there is a point $s_2\in [s_1,\ell]$ such that
\bann
\frac{1}{H(\gamma(s_2))}\geq h'(s_2)=\frac{h(X)-h_1}{\ell-s_1}\geq \frac{h(X)}{2\ell}>\sqrt{\frac{16}{4}C_1h(X)}\ge\sqrt{4C_1h(\gamma(s_2))}\,.
\eann
This contradicts \eqref{eq:Hlowerbond1} and we conclude that
\[
\Vert\proj_{\R^n\times\{0\}}X\Vert\geq\sqrt{\frac{h(X)}{16C_1}}
\]
for all $X\in M^n$ with height $h(X)\geq h_0:=2h_1$.
\end{proof}

Next, we find at each height $h$ a point with $H\sim h^{-\frac{1}{2}}$.
\begin{lemma}\label{lem:Hupperbound}
There exists $h_0<\infty$  such that, for any $h\ge h_0$, there is a point $X\in M\cap \overline B_{\sqrt{2nh}}(h e_{n+1})$ satisfying
\[
H(X)\le \sqrt{\frac{h_0}{h}}\,,
\]
where $\overline B_R(X)$ denotes the closed ball in $\R^{n+1}$ of radius $R$ centred at $X$.
\end{lemma}
\begin{proof}
We first show that at each height $h$, the ball of radius $\sqrt{2nh}$ centred at $h e_{n+1}$ intersects $M^n$.
\begin{claim}\label{claim:ball1}
For each height $h>0$, $\overline B_{\sqrt{2nh}}(he_{n+1})\cap M^n\neq\emptyset$.
\end{claim}
\begin{proof}
Under mean curvature flow, the tip of the translator reaches the point $he_{n+1}$ after time $t=h$. On the other hand, under mean curvature flow, the radius of $\pd B_R(he_{n+1})$ shrinks to the point $h e_{n+1}$ in time $t=R^2/2n$. Thus, if the ball $\overline B_R(he_{n+1})$ does not intersect $M^n$, the avoidance principle necessitates $R^2<2nh$.
\end{proof}

On the other hand, using Lemma \ref{lem:girthestimate}, the ball in the previous claim can be scaled so that it no longer intersects $M^n$.
\begin{claim}\label{claim:ball2}
There exists $0<h_0<\infty$ such that $B_{\sqrt{h/h_0}}(h e_{n+1})\cap M^n=\emptyset$ for all $h\geq h_0$. 
\end{claim}
\begin{proof}
By Lemma \ref{lem:girthestimate}, there exists $h_1>0$ such that $\Vert\proj_{\R^n\times\{0\}}X\Vert\geq\sqrt{\frac{h(X)}{16C_1}}$ for any $X\in M^n$ with height $h(X)$ at least $h_1$. Set $h_2:=\max\{1,2h_1\}$, $R:=\d\sqrt h$, where $\d:=\frac{1}{2}\min\{\frac{1}{\sqrt{16C_1}},1\}$, and consider, for any $h\ge h_2$, the cylinder $Q_R$ centred at the point $he_{n+1}$ with radius $R$ and height $2R$. Then, for any $X\in Q_R\cap M^n$, we have $h(X)\geq h-R\geq h_1$, so that
\[
R\geq \|\proj_{\R^n\times\{0\}}(X)\|\geq \sqrt{\frac{h(X)}{16C_1}}\geq\sqrt{\frac{h-R}{16C_1}}\,.
\]
Rearranging, this becomes
\[
(1-16C_1\d^2)\sqrt{h}\leq\delta\,.
\]
But this implies $h\le 2/3$. To avoid a contradiction, we must conclude that $Q_R\cap M^n=\emptyset$. The claim then follows with $h_0:=\max\{h_2,\d^{-2}\}$.

\end{proof}
By Claim \ref{claim:ball2}, there exists $h_1<\infty$ such that $B_{n\sqrt{h/h_1}}(h e_{n+1})\cap M^n=\emptyset$ for all $h\geq \frac{h_1}{n^2}$. Set $h_0:=\max\{h_1,\frac{n}{2}\}$ and define, for any $h\ge h_0$,
\[
\rho=\inf\{r:B_{r\sqrt h}(h e_{n+1})\cap M\ne\emptyset\}.
\]
By Claims \ref{claim:ball1} and \ref{claim:ball2}, we know that $\frac{n}{\sqrt{h_0}}\le \r\le \sqrt{2n}$. Thus, there exists a point $X\in \ov B_{\r\sqrt h}(h e_{n+1})\cap M^n$. Since $M^n$ and $B_{\r\sqrt h}(he_{n+1})$ are tangent at $X$ and $M^n$ lies outside of $ B_{\r\sqrt h}(he_{n+1})$, we have $H(X)\le\frac{n}{\rho\sqrt h}\le \sqrt{\frac{h_0}{h}}$.
\end{proof}

Finally, we need to show that $\kappa_1/H$ goes to zero as $h\to\infty$.

\begin{lemma}[Asymptotics for $\frac{\kappa_1}{H}$]\label{lem:kappa1bound}
For any sequence of points $X_j$ with $h(X_j)\to\infty$,
\[
\frac{\kappa_1}{H}(X_j)\to 0\,.
\]
\end{lemma}
\begin{proof}
Suppose there exists a sequence of points $X_j\in M^n$ with $h_j:=h(X_j)\to\infty$ but $\limsup_{j\to\infty}\frac{\kappa_1}{H}(X_j)>0$. Passing to a subsequence, we can assume that $\liminf_{j\to\infty}\frac{\kappa_1}{H}(X_j)>0$. We may choose another sequence of points $Y_j\in M^n$ such that 
\[
\frac{\kappa_1}{H}(Y_j)=\min_{U_j}\frac{\kappa_1}{H}\,,
\]
where $U_j:=\{X\in M^n:h(X)\leq h_j\}$. Since $M^n$ is non-compact, we know that $\frac{\kappa_1}{H}(Y_j)\to 0$ \cite{Hm94}. Moreover, the strong maximum principle implies $h(Y_j)=h_j\to\infty$ since, combining \eqref{eq:LaplaceA} and \eqref{eq:LaplaceH}, the tensor $Z:=A/H$ satisfies
\bann
-\Delta Z(u,u)=\cd_VZ(u,u)+2\inner{\cd Z(u,u)}{\frac{\cd H}{H}}\,.
\eann

Now set $\lambda_j:=H(Y_j)$ and consider the sequence $M^n_j:=\lambda_j(M^n-Y_j)$. Then
\[
H_j(0)= 1\quad \text{ and }\quad \frac{\kappa^j_1}{H_j}(0)\to 0  
\,,
\]
where $H_j$ and $\k_1^j$ are the mean curvature and smallest principal curvature, respectively, of $M^n_j$. It now follows from the gradient estimate \eqref{eq:gradient} (see Corollary \ref{cor:convergence}) that, after passing to a subsequence, the sequence $M^n_j\cap B_j$ converges locally uniformly in $C^2$ to a non-empty limit $M^n_\infty\cap B_\infty$, where $B_j$ is the intrinsic ball in $M^n_j$ of radius $(2C_1)^{-1}$ about the origin. But since the sequence $M^n_j\cap B_j$ satisfies
\[
H_j(X)=\lambda_j^{-1}\inner{\nu_j(X)}{e_{n+1}}\,,
\]
where $\nu_j(X)$ is the normal to $M^n_j$, the limit $M^n_\infty\cap B_\infty$ satisfies
\ba\label{eq:axis}
\inner{\nu_\infty}{e_{n+1}}\equiv 0\,,
\ea
where $\nu_\infty$ is the normal to $M^n_\infty$. In particular, $\kappa^\infty_1\equiv 0$ in $B_\infty$, and we conclude from the cylindrical estimate that $M^n_\infty\cap B_\infty$ lies in a cylinder of radius $(n-1)$. But this implies that the ratio $|\cd H_j|/H_j^2$ goes to zero on all of $B_j$, and, iterating Corollary \ref{cor:convergence} and passing to a diagonal subsequence, we deduce that $M^n_j$ converges locally uniformly in $C^2$ to a round orthogonal cylinder of radius $(n-1)$. Moreover, by \eqref{eq:axis}, the axis of the cylinder is parallel to $e_{n+1}$. By compactness of the constant height slices, a subsequence of $X_j$ converges to a point in the limit of height zero. But this contradicts $\liminf_{j\to\infty}\frac{\kappa_1}{H}(X_j)>0$. 
\end{proof}

We are now ready to prove that the blow-down of our translator is the shrinking cylinder.
\begin{lemma}
Denote by $M_t^n:=M^n+te_{n+1}$. Given any sequence $h_j\to\infty$, the sequence $\{M_{t,j}^n\}_{j=1}^\infty$ of mean curvature flows
\ba\label{eq:rescaledMCF}
M_{t,j}^n:=h_j^{-\frac{1}{2}}\lb M^n_{h_jt}-h_je_{n+1}\rb\,,\quad t\in(-\infty,1)
\ea
converges locally uniformly in $C^\infty$ to the shrinking cylinder $S^{n-1}_{r(t)}\times \R$, where $r(t):=\sqrt{2(n-1)(1-t)}$.
\end{lemma}
\begin{proof}
By Lemmas \ref{lem:Hto0}, \ref{lem:Hlowerbound}, \ref{lem:girthestimate} and \ref{lem:Hupperbound}, there is a sequence $X_j\in M^n$ with $h_j:=h(X_j)\to\infty$, $H(X_j)\sim h_j^{-\frac{1}{2}}$ and $\proj_{\R^n\times\{0\}}X_j\sim h_j^{-\frac{1}{2}}$. Moreover, by Lemma \ref{lem:kappa1bound}, $\frac{\kappa_1}{H}(X_j)\to 0$. As in the proof of Lemma \ref{lem:kappa1bound}, we can use Corollary \ref{cor:convergence} and the cylindrical estimate to deduce that, after passing to a subsequence, $M^n_j:=h_j^{-\frac{1}{2}}(M^n-h_je_{n+1})$ converges locally uniformly in $C^2$ to a limit $M^n_\infty$ which is congruent to a round, orthogonal cylinder. Since the limit encloses the ray $\{se_{n+1}:s>0\}$, its axis must be parallel to $e_{n+1}$. It follows that $H\sim h^{-\frac{1}{2}}$. We can now conclude, by the same argument, that for any sequence $\lambda_j\to\infty$ and any $R>0$, the sequence
\ba\label{eq:convergetocylinder}
M_{j}^n=R\lambda_j^{-\frac{1}{2}}\lb M^n-\lambda_je_{n+1}\rb
\ea
converges subsequentially to a round orthogonal cylinder with axis parallel to $e_{n+1}$. Setting $\lambda_j:=R^2h_j$ where $R:=\sqrt{1-t}$, and applying standard regularity theory (see \cite{EcHu91} or \cite{Eck}), we deduce, after passing to a subsequence, that \eqref{eq:rescaledMCF} converges locally uniformly in $C^\infty$ to a shrinking cylinder with axis parallel to $e_{n+1}$. It is also clear from \eqref{eq:convergetocylinder} that the radius of the limit goes to zero as $t\to 1$. We conclude that the limit is $S^{n-1}_{r(t)}\times \R$. 
Since the limit is the same for any convergent subsequence, the convergence holds for the entire sequence.
\end{proof}

\begin{cor}[Asymptotics for $H$]
\[
H=\sqrt{\frac{n-1}{2}}h^{-\frac{1}{2}}+o\left(h^{-\frac{1}{2}}\right)\,.
\]
\end{cor}

The remainder of the proof of Theorem \ref{thm:bowl} now follows from the work of Haslhofer \cite[Sections 3--5]{Ha}.

Finally, we show, briefly, how Corollary \ref{cor:bowl} follows from Theorem \ref{thm:bowl}.
\begin{proof}[Proof of Corollary \ref{cor:bowl}]
The cylindrical estimate follows immediately from the cylindrical estimate of Huisken and Sinestrari \cite[Theorem 1.5]{HuSi09},
\[
|A|^2-\frac{1}{n-1}H^2\leq \varepsilon H^2+C_\varepsilon\,,
\]
as the lower order term is annihilated under the rescaling no matter how small we take $\varepsilon$ (see \cite[Section 4]{HuSi99a}). The strong maximum principle then gives the strict inequality.

The gradient estimate follows from the gradient estimate of Huisken and Sinestrari \cite[Theorem 6.1 and Remark 6.2]{HuSi09},
\[
|\cd A|^2\leq Cg_1g_2\,,
\]
where $g_1:=2C_\varepsilon+\varepsilon H^2-\left(|A|^2-\frac{1}{n-1}H^2\right)$ for arbitrary $\varepsilon>0$ and $g_2=C_\delta+\delta H^2-\left(|A|^2-\frac{1}{n-1}H^2\right)$ with $\delta=\delta(n)$ fixed. In particular, $g_1\leq c_nH^2+C_n$, so that
\[
\frac{|\cd A|^2}{H^4}\leq C\lb c_n+\frac{C_n}{H^2}\rb\lb\frac{2C_\varepsilon}{H^2}+\varepsilon -\frac{|A|^2-\frac{1}{n-1}H^2}{H^2}\rb\,.
\]
Under the rescaling, all lower order terms are annihilated, and the claim follows, as for the cylindrical estimate, by taking $\varepsilon\to 0$.
\end{proof}

\section{Flows by non-linear functions of curvature}\label{sec:F}

We now consider solutions of \eqref{eq:FT} and prove Theorem \ref{thm:bowlF} and Corollary~\ref{cor:bowlF}. Let us begin with a discussion of the conditions (1)--(2) of Theorem~\ref{thm:bowlF} which will replace the corresponding conditions in Theorem \ref{thm:bowl}. 

\subsection{Flows by convex speeds} For speeds $F=f(\vec\kappa)$ given by convex admissible $f:\Gamma^n\to\R$, the cylindrical estimate takes the form
\ba\label{eq:cylindricalFconvex}
\kappa_1+\kappa_2-\beta_1^{-1}F>0\,,
\ea
where $\beta_1=f(0,1,\dots,1)$ is the value $F$ takes on the cylinder $\R\times S^{n-1}_1$. We claim that $\kappa_1$ is bounded from below by $\kappa_1+\kappa_2-\beta_1^{-1}F$, and that the only points at which both $\kappa_1$ and $\kappa_1+\kappa_2-\beta_1^{-1}F$ vanish are the cylindrical points: $\kappa_1=0$, $\kappa_2=\kappa_n$.
\begin{claim}\label{claim:cylindricalconvexF}
Set 
\[
\Lambda:=\left\{z\in\Gamma^n:\min_{1\leq i<j\leq n}\left(z_i+z_j\right)-\beta_1^{-1}f(z)>0\right\}\,.
\]
Then
\ben
\item $\Lambda\subset \Gamma_+$, and
\item $\pd\Lambda\cap \pd\Gamma_+=\cup_{\sigma\in P_n}\{k(\lambda_{\sigma(1)},\lambda_{\sigma(2)},\dots,\lambda_{\sigma(n)}):k\geq 0\}$, where $\lambda_1=0$ and $\lambda_2=\dots=\lambda_n=1$ and $P_n$ denotes the set of permutations of the set $\{1,\dots,n\}$.
\een
In particular, there is a constant $\beta_2>0$ such that
\[
\min_{1\leq i<j\leq n}\left(z_i+z_j\right)-\beta_1^{-1}f(z)\leq \beta_2\min_{1\leq i\leq n}z_i\,.
\]
\end{claim}
\begin{proof}
Note that, as a super-level set of a concave function, $\Lambda$ is convex. Note also that $(0,1,\dots,1)\in\pd\Lambda$. Thus, by symmetry and convexity, we have $(1,\dots,1)\in\Lambda$. Finally, by homogeneity and \emph{strict} monotonicity of $f$, the only points in $\overline\Lambda$ of the form $(0,z_2,\dots,z_n)$ for $0<z_i$ are those with $z_2=\dots=z_n$. Claims (1) and (2) follow. The existence of $\beta_2$ then follows from compactness of the set $\Lambda\cap\{\Vert z\Vert=1\}$ and homogeneity of $f$.
\end{proof}

The gradient estimate then takes the form
\ba\label{eq:gradientFconvex}
\frac{|\cd A|^2}{F^4}
\leq C_1\frac{\kappa_1}{F}\,.
\ea

We remark that \eqref{eq:cylindricalFconvex} holds on blow-up limits of two-convex flows by convex admissible speeds \cite{AnLa14}; however, it is unknown (to the authors) whether a gradient estimate of the form \eqref{eq:gradientFconvex} holds, except when $F$ is the mean curvature. We note that, by a similar argument as in Claim \ref{claim:gradient} below, the estimate
\bann
\frac{|\cd A|^2}{F^4}\leq C_1
\eann
would suffice.

\subsection{Flows by concave speeds} For speeds $F=f(\vec\kappa)$ given by concave admissible $f:\Gamma^n\to\R$, the cylindrical estimate takes the form
\ba\label{eq:cylindricalFconcave1}
\kappa_n-\beta_1^{-1}F<0\,,
\ea
where $\beta_1=f(0,1,\dots,1)$ is the value $F$ takes on the cylinder $\R\times S^{n-1}_1$. We claim that $\kappa_1$ is bounded from below by $\beta_1^{-1}F-\kappa_n$, and that the only points at which both $\kappa_1$ and $\kappa_n-\beta_1^{-1}F$ vanish are the cylindrical points (cf. \cite[Proposition 3.6 and Theorem 3.8]{BrHu2}).
\begin{claim}\label{claim:cylindricalconcaveF}
Set 
\[
\Lambda:=\{z\in\Gamma^n:\max_{1\leq i\leq n} z_i-\beta_1^{-1}f(z)<0\}\,.
\]
Then
\ben
\item $\Lambda\subset \Gamma_+$, and
\item $\pd\Lambda\cap \pd\Gamma_+=\cup_{\sigma\in P_n}\{k(\lambda_{\sigma(1)},\lambda_{\sigma(2)},\dots,\lambda_{\sigma(n)}):k\geq 0\}$, where $\lambda_1=0$ and $\lambda_2=\dots=\lambda_n=1$ and $P_n$ denotes the set of permutations of the set $\{1,\dots,n\}$.
\een
In particular, there is a constant $\beta_2>0$ such that
\[
\beta_1^{-1}f(z)-\max_{1\leq i\leq n}z_i\leq \beta_2\min_{1\leq i\leq n}z_i\,.
\]
\end{claim}
\begin{proof}
The proof is the same as the proof of Claim \ref{claim:cylindricalconvexF}.
\end{proof}

The gradient estimate then takes the form
\ba\label{eq:gradientFconcave1}
\frac{|\cd A|^2}{F^4}\leq C_1\frac{\kappa_1}{F}
\ea

We remark that \eqref{eq:cylindricalFconcave1} holds on blow-up limits of two-convex flows by concave admissible speeds \cite{LaLy} (cf. \cite[Theorem 3.1]{BrHu2}). Moreover, making use of \cite[Theorem 6.1]{BrHu2}, we find that the gradient estimate also holds if the underlying flow is embedded.

\begin{claim}\label{claim:gradient}
Let $X:M^n\times[0,T)\to\R^{n+1}$ be an embedded solution of \eqref{eq:F}, where $F$ is given by $F=f(\vec\kappa)$ for some admissible $f:\Gamma^n\to\R$ satisfying the conditions of Corollary \ref{cor:bowlF}. Then there is a constant $C=C(n,M_0)$ and, for any $\varepsilon>0$, a constant $F_\varepsilon=F(\varepsilon,n,M_0)$ such that
\[
\frac{|\cd A|^2}{F^4}\leq \varepsilon+C\frac{\kappa_1}{F}
\]
wherever $F>F_\varepsilon$. 
\end{claim}
\begin{proof}
We will make use of the gradient estimate of \cite[Theorem 6.1]{BrHu2}, which provides a constant $\Lambda=\Lambda(n,M_0)$ such that
\ba\label{eq:BrHugradient}
|\cd A|^2\leq\Lambda F^4\,.
\ea
We note that the interior non-collapsing estimate \cite{ALM13} and Sections 5 and 6 of \cite{BrHu2} apply to embedded flows satisfying the conditions of Corollary \ref{cor:bowlF}. 

So suppose that the claim does not hold. Then there is a constant $\varepsilon_0>0$ and a sequence of points $(x_j,t_j)\in M^n\times[0,T)$ with $F(x_j,t_j)\to\infty$ such that
\[
\frac{|\cd A|^2}{F^4}(x_j,t_j)>\varepsilon_0+j\frac{\kappa_1}{F}(x_j,t_j)\,.
\]
If 
\[
\limsup_{j\to\infty}\frac{\kappa_1}{F}(x_j,t_j)>0
\]
then we would obtain a contradiction to \eqref{eq:BrHugradient}. Otherwise, passing to a subsequence, translating in space and time, and rescaling by $\lambda_j:=F(x_j,t_j)$, we obtain a sequence of flows $X_j:M^n\times(-\lambda_j^2t_j,0]\to\R^{n+1}$ with
\bann
X_j(x_j,0)=0\,,\quad F_j(x_j,0)=1\,,\quad\kappa_1(x_j,0)\to 0\quad\text{and}\quad|\cd A|^2(x_j,0)>\frac{\varepsilon_0}{2}\,.
\eann
By \cite[Theorem 6.1]{BrHu2}, this sequence converges in a uniform parabolic neighbourhood of $(x_j,0)$ locally uniformly in $C^2$ to some non-empty smooth limit flow. By the cylindrical estimate \cite{LaLy} (cf. \cite[Theorem 3.1]{BrHu2}), this limit must satisfy $\kappa_n-\beta_1^{-1}F\leq 0$. By Claim \ref{claim:cylindricalconcaveF}, this implies $\kappa_1\geq 0$. Since $\kappa_1$ reaches zero at the origin, we can now conclude from the splitting theorem and Claim~\ref{claim:cylindricalconcaveF} that the limit is contained in a shrinking cylinder. But this contradicts the fact that $\frac{|\cd A|^2}{F^4}\ge \frac{\varepsilon_0}{2}$ at some point on the limit.
\end{proof}
It follows that blow-up limits of \eqref{eq:F} with speeds satisfying the conditions of Corollary \ref{cor:bowlF} satisfy
\[
\frac{|\cd A|^2}{F^4}\leq C\frac{\kappa_1}{F}\,.
\]

Note that flows by concave speeds are interior non-collapsing \cite{ALM13}. Moreover, the non-collapsing estimate improves at a singularity \cite{LaLy}. Thus, we can replace the cylindrical estimate by
\bann
\overline k-\beta_1^{-1}F<0\,.
\eann
This formulation of the cylindrical estimate is non-trivial in dimension $n=2$, but stronger than \eqref{eq:cylindricalFconcave1} when $n\geq 3$.

Armed with these facts, and the splitting theorem of the Appendix, we can proceed almost exactly as in Section \ref{sec:proof} to show (assuming, without loss of generality, that $f(0,1,\dots,1)=n-1$), that the blow-down of $M^n_t:=M^n+te_{n+1}$ is the shrinking cylinder $S^{n-1}_{\sqrt{2(n-1)(1-t)}}\times \R$. 

By the conditions on $F$, the remainder of the proof differs only slightly from \cite[Sections 3-5]{Ha}. Indeed, the linearization of \eqref{eq:F} is the equation
\ba\label{eq:LF}
(\pd_t-\Delta_F)u=|A|^2_Fu\,,
\ea
where, in an orthonormal frame of eigenvectors for $A$, $\Delta_F:=\frac{\pd f}{\pd \kappa_{i}}\cd_i\cd_i$ and $|A|_F:=\frac{\pd f}{\pd \kappa_{i}}\kappa_i^2$. Solutions of the linearized flow on a translating solution of \eqref{eq:F} correspond to solutions of the linearized translator equation
\ba\label{eq:LFT}
-\Delta_Fu=\cd_Vu+|A|^2_Fu
\ea
on the corresponding solution of \eqref{eq:FT}. Since the speed $F$ satisfies this equation, the strong maximum principle implies that
\bann
\sup_{h\leq h_0}\frac{\vert u\vert}{F}\leq \sup_{h=h_0}\frac{\vert u\vert}{F}
\eann
for any $u$ satisfying \eqref{eq:LFT} on a strictly convex solution of \eqref{eq:FT}.

By the invariance of \eqref{eq:F} under ambient isometries, the functions
\[
u_{J,O}(x,t):=\inner{J(X(x,t)-O)}{\nu(x,t)}
\]
satisfy \eqref{eq:LF} for any rotation generator $J\in \mathfrak{so}(n+1)$ and translation generator $O\in\R^{n+1}$. 

Recalling that we have normalized $f$ so that $f(0,1,\dots,1)=n-1$, observe that (modulo a time-dependent tangential reparametrization) the shrinking cylinders
\bann
C:S^{n-1}\times\R\times (-\infty,1){}&\to S^{n-1}_{r(t)}\times\R\subset \R^{n+1}\\
(\vartheta,h,t){}&\mapsto \lb r(t)\vartheta,h\rb
\eann
with $r(t):=\sqrt{2(n-1)(1-t)}$ satisfy \eqref{eq:F}. By symmetry and homogeneity of $f$, we find, for each $j=2,\dots,n$, that
\[
\frac{\pd f}{\pd\kappa_j}=\frac{r}{n-1}\sum_{i=2}^n\frac{\pd f}{\pd\kappa_i}\kappa_i=\frac{r}{n-1}\sum_{i=1}^n\frac{\pd f}{\pd\kappa_i}\kappa_i=\frac{r}{n-1}F=1
\]
on the shrinking cylinder, so that
\[
\Delta_F=\frac{\pd f}{\pd\kappa_1}\cd_h\cd_h+\frac{1}{r^2}\Delta_{S^n}
\]
and
\[
|A|^2_F=\frac{1}{2(1-t)}\,.
\]
It is now clear that the decay estimate \cite[Proposition 4.1]{Ha} and the contradiction argument in \cite[Section 5]{Ha} apply in the non-linear setting. This proves Theorem \ref{thm:bowlF}. Corollary \ref{cor:bowlF} then follows, since, by \cite{LaLy} (cf. \cite[Theorem 3.1]{BrHu2}) and Claim \ref{claim:gradient}, the assumptions of Theorem \ref{thm:bowlF} hold on blow-up limits of solutions of \eqref{eq:F}.

\section{Appendix: The splitting theorem}\label{sec:SMP}

We include here a proof of the splitting theorem for solutions of \eqref{eq:F}.

\begin{theorem}[Splitting Theorem]\label{thm:splittingF}
Let $X:M^n\times(0,t_0]\to\R^{n+1}$, $n\geq 2$, be a weakly convex solution of \eqref{eq:F}, where $F$ is given by $F(x)=f(\vec\kappa(x))$ for some admissible $f:\Gamma^n\to\R$ such that 
\bi
\item[(i)] $\displaystyle \{(0,\hat z):\hat z\in \Gamma^{n-1}_+\}\subset \Gamma^n$ and the function $f_\ast:\Gamma_+^{n-1}\to\R$ defined by
\bann
f_\ast(z_2^{-1},\dots,z_n^{-1}):=f(0,z_2,\dots,z_n)^{-1}
\eann
is concave. 
\ei
Suppose also that 
\bi
\item[(ii)] $\vec\kappa(M^n\times(0,t_0])\subset \overline \Gamma{}_0^n$ for some cone $\Gamma_0^n$ satisfying $\overline\Gamma{}_0^n\setminus\{0\}\subset \Gamma_2^n$, where $\Gamma_2^n:=\{z\in\R^n:\min_{1\leq i<j\leq n}\{z_i+z_j\}>0\}$. 
\ei
Then $\kappa_1(x_0,t_0)=0$ for some $x_0\in M^n$ only if $\kappa_1\equiv 0$ and $M^n$ splits isometrically as a product $M^n\cong \R\times\Sigma^{n-1}$.
\end{theorem}
\begin{proof}
This was established for convex speeds in \cite[Theorem 4.21]{La}. The proof for speeds satisfying the weaker inverse-concavity condition is similar:

Suppose that $\kappa_1$ reaches zero at an interior space-time point $(x_0,t_0)$. By hypothesis, $\kappa_1<\kappa_2$ at this point. Let $U$ be the largest space-time neighbourhood of $(x_0,t_0)$ in $M^n\times(0,t_0]$ such that $\kappa_1<\kappa_2$. Then $U$ is open, $\kappa_1$ has a unique principal direction field $e_1$ in $U$, and both are smooth in $U$. 

Differentiating $\kappa_1=A(e_1,e_1)$ yields
\bann
\cd_k\kappa_1=\cd_kA_{11}+2A(\cd_ke_1,e_1)\,,
\eann
so that
\ba\label{eq:Dk1}
\cd_kA_{11}=\cd_k\kappa_1=0
\ea
at $(x_0,t_0)$ for each $k$. Note that $\cd_ke_1\perp e_1$ since $e_1$ has constant length. Differentiating the eigenvalue identity $A(e_1)=\kappa_1e_1$ yields the remaining components:
\bann
(A-\kappa_1I)(\cd_ke_1)=\lb\cd_k\kappa_1 I-\cd_kA\rb(e_1)\,,
\eann
so that
\ba\label{eq:De1}
\cd_ke_1=-R\lb\cd_kA(e_1)\rb\,,
\ea
where $R:=(A-\kappa_1I)|_{e_1^\perp}^{-1}\circ \proj_{e_1^\perp}$. Next, consider the time derivative
\bann
\pd_t\kappa_1=\cd_tA_{11}+2A(\cd_te_1,e_1)\,,
\eann
where the covariant time derivative $\cd_t$ is defined on vector fields $v$ via $\cd_tv=[\pd_t,v]-HA(v)$, and extended to tensor fields by the Leibniz rule. This yields
\bann
\pd_t\kappa_1=\cd_tA_{11}
\eann
at $(x_0,t_0)$. Finally, we compute the Hessian,
\bann
\cd_k\cd_l\kappa_1=\cd_k\cd_lA_{11}+4\cd_kA(\cd_le_1,e_1)+2A(\cd_k\cd_le_1,e_1)+2A(\cd_ke_1,\cd_le_1)\,.
\eann
Applying \eqref{eq:De1} and the Codazzi identity, we obtain
\bann
\cd_k\cd_l\kappa_1=\cd_k\cd_lA_{11}-2R(\cd_1A(e_k),\cd_1A(e_l))
\eann
at $(x_0,t_0)$.

In an orthonormal frame of eigenvectors of $A$, we have the evolution equation \cite{An94a}
\ba\label{eq:FevolveA}
(\cd_t-\eL)A_{ij}=|A|_F^2A_{ij}+\frac{\pd^2 F}{\pd A_{pq}\pd A_{rs}}\cd_iA_{pq}\cd_jA_{rs}\,,
\ea
where $\eL:=\frac{\pd F}{\pd A_{kl}}\cd_k\cd_l$ and $|A|^2_F:=\frac{\pd F}{\pd A_{kl}}A^2_{kl}$, and we conclude
\bann
(\pd_t-\eL)\kappa_1=|A|_F^2\kappa_1+N(A,\cd A)\,,
\eann
where
\bann
N(A,\cd A):={}&2A\big((\cd_t-\Delta)e_1,e_1\big)+\frac{\pd^2F}{\pd A_{pq}\pd A_{rs}}\cd_1A_{pq}\cd_1A_{rs}\\
{}&+2\frac{\pd F}{\pd A_{kl}}\Big[2R(\cd_1A_k,\cd_1A_l)-A\big(R(\cd_1A_k),R(\cd_1A_l)\big)\Big]\,.
\eann
Observe that, at any boundary point $Z\in \mathrm{Sym}_{\Gamma^n\cap\pd\Gamma^n_+}$, the space of symmetric $n\times n$ matrices with eigenvalues $z$ in $\Gamma^n\cap\pd\Gamma^n_+$, we have, for any totally symmetric $T\in \R^n\otimes\R^n\otimes\R^n$,
\bann
N(Z,T)={}&B^p(Z,T)T_{p11}+\sum_{p,q,r,s>1}Q^{pq,rs}(Z)T_{1pq}T_{1rs}\,,
\eann
where
\bann
B^1(Z,T):={}&\left.\lb\frac{\pd^2 F}{\pd A_{11}\pd A_{11}}T_{111}+2\sum_{p,q>1}\frac{\pd^2 F}{\pd A_{pq}\pd A_{11}}T_{1pq}\rb\right|_{Z}\,,\\
B^p(Z,T):={}&R^{pq}\left.\lb\frac{\pd F}{\pd A_{11}}T_{11q}+2\sum_{k>1}\frac{\pd F}{\pd A_{1k}}T_{k1q}\rb\right|_{Z} \quad\text{for}\quad p>1
\eann
and
\bann
Q^{pq,rs}(Z):=\left.\lb\frac{\pd^2 F}{\pd A_{pq}\pd A_{rs}}+\frac{\pd F}{\pd A_{pr}}R^{qs}\rb\right|_{Z}\,.
\eann
We claim that, as quadratic forms on the space of $(n-1)\times (n-1)$ symmetric matrices,
\ba\label{eq:ICcond}
Q\geq 2\frac{DF\otimes DF}{F}
\ea
at any $Z\in \mathrm{Sym}_{\Gamma_+^n\cap\pd\Gamma^n_+}$. Indeed, embedding the space $\mathrm{Sym}_{\Gamma_+^{n-1}}$ of positive definite $(n-1)\times(n-1)$ symmetric matrices into the space $\mathrm{Sym}_{\overline \Gamma{}^n_+}$ of non-negative definite $n\times n$ symmetric matrices via the natural inclusion, the inverse-concavity condition is equivalent to concavity of the function $F_\ast:\mathrm{Sym}_{\Gamma_+^{n-1}}\to\R$ defined by $F_\ast(Z^{-1}):=F(Z)^{-1}$, where $F(Z):=f(z)$ and $z$ is the $n$-tuple of eigenvalues of $Z$. Differentiating this identity in the direction of $B\in \mathrm{Sym}_{\R^{n-1}}$, we find
\bann
-D_{X}F_\ast|_{Z^{-1}}={}&
-\frac{1}{F^2(Z)}D_BF|_{Z}\,,
\eann
where $X:=Z^{-1}BZ^{-1}$. Differentiating once more yields
\bann
D_{X}D_{X}F_\ast|_{Z^{-1}}+2D_{XZX}F_\ast|_{Z^{-1}}
={}&\frac{2}{F^3(Z)}(D_BF|_{Z})^2-\frac{1}{F^2(Z)}D_BD_BF|_{Z}
\eann
and we conclude
\bann
0\leq{}&-D_XD_XF_\ast|_{Z^{-1}}\\
={}&\frac{1}{F^2(Z)}\left.\lb D^2F-\frac{2DF\otimes DF}{F}+2DF\ast Z^{-1}\rb\right|_{Z}(B,B)\,,
\eann
where $\ast$ denotes the product $(R\ast S)^{pq,rs}:=R^{pr}S^{qs}$. This implies \eqref{eq:ICcond}. 

We now return to the evolution equation for $\kappa_1$. Note that $N$ is Lipschitz with respect to $A$. Thus, denoting by $\overline A$ the projection of $A$ onto $\pd\mathrm{Sym}_{\Gamma^n_+}$, we obtain
\bann
(\pd_t-\eL)\kappa_1+B^k\cd_k\kappa_1\geq{}&-\vert N(A,\cd A)-N(\overline A,\cd A)\vert\\
\geq{}&-C\Vert A-\overline A\Vert\\
={}&-C\kappa_1\,,
\eann
where $C$ is the worst Lipschitz constant of $N(\cdot,\cd A)$ on the set $U$. Note that $C$ is bounded on any compact subset of $U$. The strong maximum principle now implies that $\kappa_1\equiv 0$ on $K$ for any compact subset $K$ of $U$. It follows that $U\subset\{(x,t)\in M^n\times(0,t_0]:\kappa_1(x,t)=0\}\subset U$ and we deduce that $U$ is closed, and hence equal to $M^n\times (0,t_0]$. But in that case, we must have, by \eqref{eq:ICcond},
\bann
0\equiv \frac{\pd F}{\pd A_{pq}}\cd_1A_{pq}\,.
\eann
By monotonicity of $F$, we conclude that $\cd_1A\equiv 0$. 

Using standard arguments, we can now deduce the splitting: Observe that, for any $v\in \Gamma(\ker(A))$,
\bann
0\equiv \cd_k(A(v))=\cd_kA(v)+A(\cd_kv)=A(\cd_kv)\,.
\eann
Thus, $\cd_kv\in \Gamma(\ker(A))$ whenever $v\in \Gamma(\ker(A))$; that is, $\ker(A)\subset TM^n$ is invariant under parallel translation in space. Since, for any $v\in \Gamma(\ker A)$ and any $u\in TM^n$, we have
\bann
\XD_uX_\ast v= X_\ast \cd_uv-A(u,v)\nu=X_\ast \cd_uv\in X_\ast\ker A\,,
\eann
where $\XD$ is the pull-back of the Euclidean connection along $X$, we deduce that $X_\ast\ker A\subset T\R^{n+1}$ is parallel (in space) with respect to $\XD$.

Moreover, using the evolution equation \eqref{eq:FevolveA} for $A$, we obtain
\bann
\cd_tA(v)={}&\eL A(v)\\
={}&\frac{\pd F}{\pd A_{kl}}\lsb\cd_k\lb\cd_lA(v)\rb-\cd_l(A(\cd_kv))-A(\cd_l\cd_kv)\rsb\\
={}&0\,,
\eann
so that
\bann
A(\cd_tv)=\cd_t(A(v))-\cd_tA(v)=0\,;
\eann
that is, $\ker A$ is also invariant with respect to $\cd_t$. Since, for any $v\in \Gamma(\ker(A))$, we have $\cd_vF=\frac{\pd F}{\pd A_{kl}}\cd_vA_{kl}\equiv 0$, this implies that
\bann
\XD_tX_\ast v=(\cd_vF)\nu+X_\ast \cd_tv=X_\ast \cd_tv\,,
\eann
and we deduce that $X_\ast\ker A$ is also parallel in time. We conclude that the orthogonal compliment of $X_\ast\ker(A)$ is a constant (in space and time) subspace of $\R^{n+1}$.

Now consider any geodesic $\gamma:\R\to M^n\times\{t\}$, $t\in (0,t_0]$, with $\gamma'(0)\in \ker(A)$. Then, since $\ker(A)$ is invariant under parallel translation, $\gamma'(s)\in\ker(A)$ for all $s$, so that
\bann
\XD_sX_\ast\gamma'= X_\ast\cd_s\gamma'-A(\gamma',\gamma')\nu=0\,.
\eann

Thus, $X\circ\gamma$ is geodesic in $\mathbb{R}^{n+1}$. We can now conclude that $X$ splits off a line, $M^n\cong \mathbb{R}\times\Sigma^{n-1}$, such that $\mathbb{R}$ is flat ($T\R$ is spanned by $\ker(A)$) and $\Sigma^{n-1}$ is strictly convex ($T\Sigma^{n-1}$ is spanned by the rank space of $A$) and maps into the constant subspace $\lb X_\ast\ker(A)\rb^\perp\cong \mathbb{R}^{n}$.

It follows that $X\big|_{\{0\}\times\Sigma^{n-1}\times (0,t_0]}$ satisfies
\ba\label{eq:tildeF}
\pd_t\widetilde X(\widetilde x,t)=-\widetilde F(\widetilde x,t)\widetilde\nu(\widetilde x,t)\,,
\ea
for all $(\widetilde x,t)\in \{0\}\times \Sigma^{n-1}\times(0,t_0]$, where $\widetilde \nu=\nu\big|_{\{0\}\times\Sigma\times(0,T]}$ and $\widetilde F$ is given by the restriction of $f$ to $\Gamma_+^{n-1}\cong\{z\in \overline\Gamma_+:z_1=0, z_{2}>0,\dots,z_n>0\}$.
\end{proof}

\begin{rmks}\label{rem:splittingF}\mbox{}
\begin{enumerate}
\item By condition (ii), the cross-section $\Sigma^{n-1}$ in the splitting must be compact \cite{Hm94}, and we conclude, by uniqueness of solutions of \eqref{eq:tildeF}, that the isometric splitting persists until the maximum time.
\item Flows by convex admissible speeds defined on the faces of $\Gamma^n_+$ automatically satisfy condition (i).
\item If $n=2$, flows by admissible speeds defined on the faces of $\Gamma{}^2_+$ automatically satisfy condition (i).
\item \label{rem:uniform} Condition (ii) can be arranged if the flow preserves any form of \emph{uniform} two-convexity. This is the case for flows by convex speeds, which preserve $\kappa_1+\kappa_2\geq\alpha F$, flows of surfaces (trivially) and flows by concave speeds satisfying $f\big|_{\pd\Gamma^n}\equiv 0$, $\Gamma^n\subset\Gamma_2^n$, which preserve $\kappa_n\leq CF$ or $H\leq CF$.
\end{enumerate}
\end{rmks}

\bibliographystyle{plain}
\bibliography{bibliography}

\end{document}